\documentclass[letterpaper,10pt]{amsart}
\usepackage{amsmath,amssymb,amsxtra,amsthm, amstext, amscd,amsfonts,fancyhdr, hyperref,enumerate, xcolor}
\usepackage[OT2,T1]{fontenc}
\DeclareSymbolFont{cyrletters}{OT2}{wncyr}{m}{n}
\DeclareMathSymbol{\Sha}{\mathalpha}{cyrletters}{"58}

\newcommand{\bQ}{{\mathbb{Q}}}
\newcommand{\bR}{{\mathbb{R}}}

\newcommand{\bZ}{{\mathbb{Z}}}

  \newcommand{\A}{{\mathcal{A}}}

\newcommand{\Li}{\operatorname{Li}}
\newcommand{\Art}{\operatorname{Art}}
\newcommand{\Disc}{\operatorname{Disc}}
\newcommand{\Stab}{\operatorname{Stab}}

\newcommand{\ep}{\varepsilon}

\newcommand{\upchi}{{\raise.35ex\hbox{$\chi$}}}

\newcommand{\SL}{\operatorname{SL}}

\newcommand{\Kronecker}[2]{\genfrac{(}{)}{}{}{#1}{#2}}

\newcommand{\Hur}{\mathrm{Hur}}

\makeatletter
\@namedef{subjclassname@2020}{%
  \textup{2020} Mathematics Subject Classification}
\makeatother

\newtheorem{theorem}{Theorem}[section]
\newtheorem{corollary}[theorem]{Corollary}
\newtheorem{proposition}[theorem]{Proposition}
\newtheorem{lemma}[theorem]{Lemma}

\theoremstyle{definition}
\newtheorem{definition}[theorem]{Definition}

\newtheorem{remark}[theorem]{Remark}

\numberwithin{equation}{section}

\begin{document}

\title{On the number of binary quadratic forms having discriminant $1-4p$, $p$ prime}

\author{Alison Beth Miller}
\address{Mathematical Reviews \\
535 W. William St, Suite 210\\
Ann Arbor, MI 48103 }
\email{alimil@umich.edu}

\author{Stanley Yao Xiao}
\address{Department of Mathematics and Statistics \\
University of Northern British Columbia \\
3333 University Way \\
Prince George, British Columbia, Canada V2N 4Z9 }
\email{syxiao@math.toronto.edu}
\indent


\begin{abstract} 
In this paper we obtain an asymptotic formula for the number of $\SL_2(\bZ)$-equivalence classes of positive definite binary quadratic forms over $\bZ$ having bounded discriminant $\Delta = 1-4p$, with $p$ a prime.  We also give a random Euler product model for the distribution of Hurwitz class numbers, which is supported by our formula. 
\end{abstract}

\maketitle

\section{Introduction}
\label{Intro}

In \cite{Mil}, the first named author gave an account of a connection between simple $(4a+1)$-knots of genus $1$ with Alexander polynomial
\begin{equation} \label{Deltam} \Delta_m(t) = mt^2 - (2m-1)t + m, m \in \bZ,\end{equation}
and binary quadratic forms of discriminant $1-4m$.  In particular, there is a surjective map from such binary quadratic forms to the aforementioned knots, which is injective exactly when $m = \pm p$ is a prime or negative of a prime, and otherwise is expected to have large kernel. As a consequence, the first named author gave some heuristics
suggesting that knots with Alexander polynomial $\Delta_p(t)$ with $p$ prime should dominate the total count (up to isomorphism) of $(4a+1)$-knots of genus $1$ with Alexander polynomial $1-4m$ with $m \in [0, X]$ not necessarily prime.
Determining the size of the heuristically dominant term leads to the following question, which is purely about definite binary quadratic forms: \\ \\
\textbf{Question}: How many $\SL_2(\bZ)$-classes of integral, positive definite binary quadratic forms are there whose discriminant is of the form $1-4p$ with $p$ prime, and bounded by $X$? \\ 


Note that here the binary quadratic forms are allowed to be imprimitive.  For knots of genus 1, primitivity of the quadratic form associated to a knot is equivalent to cyclicity of the first homology group of the double branched cover. For higher genus knots there is an analogous condition for primitivity in terms of the structure of this first homology group, but it is less pleasant.  Although in the genus 1 case this topological condition is relatively nice, it is not particularly natural to restrict to such knots. \\

Put $H(D)$ for the total number of $\SL_2$-equivalence classes of positive definite binary quadratic forms of discriminant $D$. This is almost the same as the Hurwitz class number, except that we are not weighting the forms that have automorphisms, and it has the same asymptotics (see Lemma~\ref{H minus Hur}  and Corollary~\ref{difference with Hur}). 
Our question is then equivalent to: what is $\sum_{p \leq X} H(1 - 4p)$?\\


The previous paper \cite{Mil} conjectured that the answer should be asymptotic to $X^{3/2}/\log X$, based on the naive heuristic assumption that class numbers $H(1-4p)$ ranging over $p$ prime, have the same statistical behaviour as the class numbers $H(1-4m)$ where $m$ ranges over all positive integers. There, a sieve was applied to prove an upper bound of the form $O(X^{3/2}/\log X)$, without determining the constant explicitly.  In this paper we give an exact asymptotic for the counting problem. Our strongest version of this asymptotic, which has an arbitrary log power savings, can be found in Theorem~\ref{main result CNF}, but we give the following simpler version here:

\begin{theorem} \label{MT} For $X \ge 2$, we have the asymptotic formula
\[\sum_{p \leq X} H(1 - 4p) = C_{\Art} \cdot \frac{2 \pi}{9} \cdot \frac{X^{3/2}}{\log X} + O \left(\frac{X^{3/2}} {(\log X)^{2}} \right). \]
\end{theorem}
We note that Theorem \ref{MT} follows from the more precise Theorem \ref{main result CNF}; see (\ref{MTIBP}). \\

Here the constant $C_{\Art}$ is \emph{Artin's constant}, given by the Euler product
\[C_{\Art} = \prod_{\ell \ge 2} \left(1 - \frac{1}{\ell(\ell-1)} \right). \]
This shows that the heuristic assumption in \cite{Mil} is not quite right.  Comparing with the asymptotic \[{\sum_{m \le X} H(1-4m)} \sim \frac{2\pi}{9} X^{3/2}\] 
as $X \to \infty$, we see that the Hurwitz class numbers $H(1-4p)$ 
are smaller on average than the Hurwitz class numbers $H(1-4m)$ by a factor of Artin's constant $C_{\Art}$, even though the Hurwitz class numbers $H(-p)$ with $p \equiv 3 \pmod 4$ are known to have the same average behavior as $H(1-4m)$ \cite{Nagoshi}.  We explain this discrepancy in Section~\ref{section:randomEuler} in terms of random Euler products; this model could also be used to predict higher moments.   Although random Euler products are a standard method for modeling special values of $L$-functions, the Hurwitz class numbers $H(1-4m)$ are not special values of standard quadratic $L$-functions, but instead of ``Zagier $L$-functions'' as defined in \cite{Zagier}.  Thus the novelty of Section~\ref{section:randomEuler} lies in giving a random model for Zagier $L$-functions that can be used to study the distribution of Hurwitz, rather than regular, class numbers.
  \\


Artin's constant originally arose in the context of Artin's conjecture on primitive roots \cite{Hoo1}, namely predicting the proportion of primes $\ell$ for which a given non-square integer $a \ne - 1$ is a primitive root mod $\ell$ (this proportion is expected to be independent of $a$ provided that $a$ is not a perfect power). We do not know of a direct connection between Theorem \ref{MT} and Artin's conjecture. \\

Similar problems were considered by Friedlander and Iwaniec in \cite{FI hyperbolic}, for the case of discriminants of the form $-4p$, and by Nagoshi in \cite{Nagoshi} for discriminants of the form $-p \equiv 1 \pmod{4}$, with Nagoshi notably computing all moments.  Our proof, using the analytic class number formula, is inspired by these proofs.   However, in implementing this strategy we encounter additional complications:  where Friedlander--Iwaniec and Nagoshi obtain a single main term, we end up with infinitely many main terms, which we must sum to obtain the final answer.  These come from two sources: firstly, we are including non-primitive forms in our count, and secondly, because numbers of the form $1-4p$, unlike primes, are slightly more likely to be non-residues modulo a prime $\ell$ than they are to be residues. \\

We emphasize this point. While in previous works restricting to primes did not affect the average size of the class number, in our case it does, producing a factor of $C_{\Art}$ that does not appear in the total count of all binary quadratic forms.\\


As previously noted, the correspondence between high-dimensional knots and quadratic forms allows us to restate our results in topological terms, which we do in Section~\ref{subsec: knots}.  Since this correspondence goes by way of Seifert pairings, which are also used in classical low-dimensional knot theory, we obtain the following corollary in terms of objects of relevance to low-dimensional knot theorists. We use the notation $M^T$ to denote the transpose of a matrix $M$.

\begin{corollary}\label{cor: seifert}
The number of $S$-equivalence classes of $2 \times 2$ Seifert matrices $P$ with Alexander polynomial $\det (t P - P^T) = p t^2 + (1-2p) t + p$ for $p \in [1, X]$ prime is asymptotic to $C_{\Art} \frac{4\pi}{9} X^{3/2}/(\log X)$ as $X \to \infty$.   
\end{corollary}

\begin{remark}
    By a classical result of Gauss we have that $H(1-4p)$ is equal to 24 times the number of ways of writing $4p-1$ as a sum of three squares (provided that $p$ is odd and so $4p-1$ is 3 mod $4$).  That is, our result can also be interpreted as counting lattice points $(a, b, c)$ in a ball of radius $\sim 2 \sqrt{X}$ such that $(a^2 + b^2 + c^2-1)/4$ is a prime number.  This could provide another route to justifying our asymptotic heuristically, as the volume of this ball times the proportion of lattice points that we expect to survive sieving.  However, it is not clear how to use such sieves to prove an estimate with the strength of Theorem~\ref{MT}.
\end{remark}


\subsection{Outline of the proof}

Our proof starts from the usual analytic class number formula 
\[h(D) = \frac{1}{\pi} L(1, \chi_D) |D|^{1/2},\]
for negative discriminant $D$, where $h(D)$ denotes the class number of the quadratic field $\bQ(\sqrt{D})$ and $\chi_D = \Kronecker{D}{\cdot}$ is the primitive Dirichlet character associated to $\bQ(\sqrt{D})$. Here $\Kronecker{\cdot}{\cdot}$ is the Kronecker symbol, defined for odd  primes by 
\[\Kronecker{a}{p}  = \begin{cases} 1 & \text{if } a \text{ is a quadratic residue modulo }p \\ -1 & \text{if } a \text{ is a quadratic non-residue modulo } p \\ 0 & \text{if } p \mid  a, \end{cases}\]
 for $2$ by 
\[\Kronecker{a}{2}  = \begin{cases} 1 & \text{if } a \text{ is $\pm 1$ modulo }8 \\ -1 & \text{if } a \text{ is $\pm 3$ modulo } 8 \\ 0 & \text{if } 2 \mid  a, \end{cases}\]
and for $-1$ by
\[\Kronecker{a}{-1} = \begin{cases}
1 & \text{if }a > 0\\
-1 & \text{if }a < 0
\end{cases}
\]
and extended multiplicatively. It is known that both $\Kronecker{\cdot}{m}$ and $\Kronecker{m}{\cdot}$ for $m$ fixed are completely multiplicative functions. \\

This applies even when $D$ is not a fundamental discriminant, but the class number $h(D)$ only counts primitive binary quadratic forms.  Since we want to count all binary quadratic forms, including the imprimitive ones, our first step is to write our sum $Q(X) = \sum_{d \ge 1} Q_d(X)$,  where $Q_d(X)$ counts only the forms $ax^2 + bxy + cy^2$ with content $\gcd(a, b, c)= d$ and discriminant up to $X$ (in absolute value).
To estimate $Q_d(X)$, we first estimate the related quantity $T_d(X) = \sum_{\substack{p \le  X\\ d^2 \mid 1-4p}} L(1, \chi_{(1-4p)/d^2})$. We acknowledge that this notation is somewhat unpleasant, but it would be even more cumbersome to introduce auxiliary notation to alleviate the burden placed here. \\


This reduces our problem to bounding the following weighted bilinear character sum:
\[
T_d(X) = \sum_{\substack{p \le  X\\ d^2 \mid 1-4p}} L(1, \chi_{(1-4p)/d^2}) = \sum_{{\substack{p \le  X\\ d^2 \mid 1-4p}}} \sum_{n \ge 1} \frac{1}{n} \Kronecker{(1-4p)/d^2}{n}.
\]
  We are able to cut off the sums in the $L$-functions at $n \sim X^{1/2+\ep}$.  We then show that the range $(\log X)^B < n < X^{1/2 + \ep}$ yields an error term by dividing  into dyadic intervals and applying  bounds to the bilinear character sums over each interval.  Finally, for $n \le (\log X)^B$ we use Siegel--Walfisz to estimate the sum over $p$: this sum is generally nonzero, so we then have to combine our Siegel--Walfisz main terms into one big main term.


\subsection*{Acknowledgments}

The authors thank Manjul Bhargava for comments on an earlier draft of this paper.  We also thank the referee for helpful comments that improved the exposition of this paper as well as the quality of the results. 

\section{Motivation from and application to knot theory} 
\label{subsec: knots}

This line of research was originally motivated by a question in knot theory. We sketch the appropriate background   and state the implications of our results in terms of knots and knot invariants.  For more details, see Section 2 of \cite{Mil} for more details. \\

Roughly speaking, an $n$-knot is a ``nicely'' embedded copy of $S^n$ in $S^{n+2}$, up to topological equivalence. Here it matters that we keep track of the orientations both on $S^n$ and $S^{n+2}$. The most well-known case is that of $n=1$, where the classification of knots and study of their invariants is an extremely rich and active field.  In higher dimensions, the study of all knots only gets more complicated, but certain special families of knots are well understood, in particular the simple $n$-knots, which are those for which the first $\lfloor\frac{n}{2}\rfloor$ homotopy groups of the knot complement are ``as trivial as possible'' (explicitly, $\pi_1$ is isomorphic to $\bZ$ and $\pi_i$ is trivial for $2 \le i \le \frac{n}{2}$).\\

For $n = 5$ and $n \ge 7$, simple $n$-knots have been completely classified in terms of algebraic data: notably, this classification only depends on the dimension $n$ modulo $4$.  Simple knots have a fundamental invariant, the Alexander polynomial, and Bayer and Michel proved that for a squarefree polynomial $\Delta \in \bZ[t]$, there are only finitely many simple $n$-knots with Alexander polynomial $\Delta$.\\
 
The paper \cite{Mil} studied simple $(4a+1)$-knots of genus $1$, for fixed $a \ge 1$ (as noted above the classification does not depend on $a$).  These are exactly the simple $(4a+1)$-knots with Alexander polynomial $\Delta_m(t) = m t^2 + (1-2m) t + m$ for some nonzero integer $m$.  Using the algebraic classification of simple $(4a+1)$-knots, in \cite{Mil} the first named author showed that:

\begin{theorem}[\cite{Mil} Theorem 2.5 (vi), Corollary 2.13]
Simple $(4a+1)$-knots with Alexander polynomial $\Delta_m$ are in bijection with $\SL_2(\bZ[1/m])$-equivalence classes of binary quadratic forms over $\bZ[1/m]$ of discriminant $1-4m$. Furthermore, when $m=p$ is prime, $\SL_2(\bZ[1/p])$-equivalence classes of binary quadratic forms of discriminant $1-4m$ over $\bZ[1/m]$ are naturally identified with $\SL_2(\bZ)$-equivalence classes of binary quadratic forms  over $\bZ$ of discriminant $1-4p$.
\end{theorem}

Combining this with our Theorem~\ref{MT}, and including a factor of $2$ as we are allowing both positive and negative definite forms,

\begin{corollary}
The total number of simple $(4a+1)$-knots of genus $1$ with Alexander polynomial $p t^2 + (1-2p) t + p$ for $p \in [1, X]$ prime is $\sim C_{\Art} \frac{4\pi}{9} X^{3/2}/(\log X)$ as $X \to \infty$.
\end{corollary}

The heuristics of \cite{Mil} indicate that most simple knots of genus $1$ (ordered by the height of the Alexander polynomial) should have Alexander polynomial of this form.  That is, the total number of simple $(4a+1)$-knots of genus $1$ with Alexander polynomial $m t^2 + (1-2m) t + m$ for $m \in [-X, X]$ should also be $C_{\Art} \frac{4\pi}{9} X^{3/2}/(\log X) + O(X^{3/2}/(\log (X))^2)$.  It is an interesting question for future research to see if this can be proven: showing it will require getting some control on the sizes of the oriented class groups of the rings $\bZ \left[ \frac{1}{m}, \frac{1 + \sqrt{1-4m}}{2} \right]$, which is more complicated when $m$ is not prime. \\

Finally, we note that the algebraic invariants (the Alexander module and Blanchfield pairing, or equivalently the Seifert matrix) used to classify $(4a+1)$-knots for $a \ge 1$ are also useful in the low-dimensional case $a=0$, though they are no longer complete invariants of the knot, and miss a lot of knot-theoretic information.  Indeed, $(4a+1)$-knots for fixed $a \ge 1$ are in bijection with $S$-equivalence classes of Seifert forms \cite{Mil}[Theorem 2.4].\\

This allows us to restate the previous corollary in terms of objects of interest to low-dimensional topologists.  We obtain that the number of $S$-equivalence classes of $2 \times 2$ Seifert matrices $P$ with Alexander polynomial $\det (t P - P^T) = p t^2 + (1-2p) t + p$ for $p \in [1, X]$ prime is $\sim C_{\Art} \frac{4\pi}{9} X^{3/2}/(\log X)$ as $X \to \infty$, proving Corollary~\ref{cor: seifert} stated earlier in the introduction.

\section{A random Euler product model for Hurwitz class numbers}\label{section:randomEuler}
\subsection{Preliminaries on Hurwitz class numbers}

\begin{definition}
Let $D < 0$ be a negative discriminant.  Then we define the Hurwitz class number of discriminant $D$ by

\[
\Hur(D) = \sum_{\substack{q \bmod \SL_2(\bZ) \\ \Disc(q) = D}} \frac{2}{|\Stab_{\SL_2(\bZ)}(q)|}
\]
where the sum is over all $\SL_2(\bZ)$-equivalence classes of binary quadratic forms, and $\Stab_{\SL_2(\bZ)}(q)$ denotes the stabilizer of the quadratic form $q$ under the action of $\SL_2(\bZ)$ (which always contains the order 2 subgroup $\pm I$).  
\end{definition}

We then have: 

\begin{lemma}\label{H minus Hur}
\begin{equation} H(D) - \Hur(D) = \begin{cases}
            1/2 & \text{if }D = -4n^2 \\
            2/3 & \text{if }D = -3n^2 \\
            0 & \text{otherwise}
        \end{cases}
\end{equation}

\end{lemma}

\begin{proof}
This is an immediate consequence of the well-known fact that a positive definite, reduced binary quadratic form $q$ has $\Stab_{\SL_2(\bZ)}(q) = \pm I$ unless
\begin{itemize}
\item $q$ is $\SL_2(\bZ)$-equvialent to $n(x^2 + xy + y^2)$, in which case $\Stab_{\SL_2(\bZ)}(q)$ has order 6 and $\Disc(q) = -3n^2$, or
\item $q$ is $\SL_2(\bZ)$-equvialent to $n(x^2 + y^2)$, in which case $\Stab_{\SL_2(\bZ)}(q)$ has order 4 and $\Disc(q) = -4n^2$.
\end{itemize}
\end{proof}

As an immediate corollary, we get 
\begin{corollary}\label{difference with Hur}

    \begin{equation}
        \sum_{p \leq X} H(1 - 4p) = \sum_{p \le X} \Hur(1-4p) + O(X^{1/2})
    \end{equation}
\end{corollary}
For the rest of the paper, we will work with $\Hur(D)$ rather than $H(D)$, because it is easier to relate to values of $L$-functions.  We now outline this below. \\

Let $h(D)$ be the classical class number, where we count only primitive, reduced quadratic forms $ax^2 + bxy + cy^2$ with $\gcd(a, b,c) =1$, with the same weighting as above:
\[
h(D) = \sum_{\substack{q \bmod \SL_2(\bZ)\\ \text{primitive} \\ \Disc(q) = D}} \frac{2}{|\Stab_{\SL_2(\bZ)}(q)|}.
\]

We now state the class number formula for classical class numbers.
\begin{proposition}[Class Number Formula for negative discriminants] \label{CNF classic}
If $D$ is a negative (not necessarily fundamental) discriminant, then
\[
h(D) =  \frac{1}{\pi} L(1, \chi_{D}) |D|^{1/2}
\]
\end{proposition}

This implies the following formula for Hurwitz class numbers in terms of $L$-functions.

\begin{lemma}\label{Hurwitz and Dirichlet L}
    \[\Hur(D) = \sum_{d^2 \mid D} \frac{h(D/d^2)}{w(D/d^2)} = \sum_{d^2 \mid D} \frac{1}{\pi} L(1, \chi_{D/d^2}) d^{-1} |D|^{1/2}\]
\end{lemma}
\begin{proof}
The first equality is because any quadratic form with discriminant $D$ is equal to $d$ times a primitive form of discriminant $D/d^2$, while the second one is from Proposition~\ref{CNF classic}
\end{proof}

\subsection{The $L$-function for Hurwitz class numbers}
Lemma~\ref{Hurwitz and Dirichlet L} gives us a formula for the Hurwitz class number as a linear combination of values of the well-understood Dirichlet $L$-functions, which we will ultimately use in our proof of the asymptotic for the sum.  However, for the purposes of heuristically understanding the distribution of Hurwitz class numbers, it is convenient to have a formula that expresses the Hurwitz class number as the value of a single Euler product, and then study the distribution of each of the factors in this product.  Such a formula is known in the literature, for instance in the work of Zagier \cite{Zagier}. However, no self-contained exposition is available, so we provide one here. To this end, define 
\[
\zeta_D(s) = \zeta(2s) \sum_{n \ge 1} \# \{a \bmod 2n : a^2 \equiv D \pmod {4n}
\}n^{-s}
\]
Note that if $D$ is a fundamental discriminant then $\zeta_D(s)$ agrees with the Dedekind zeta function of the quadratic field $\bQ(\sqrt{D})$. The Dirichlet series here converges absolutely for ${\rm Re}(s) > 1$, provided that $D$ is not a square. \\

The following formulation of $\zeta_D(s)$ makes the relation to binary quadratic forms clear.
\begin{proposition}\label{zeta formula}
If $D < 0$, then
 \[   \zeta_D(s) = \sum_{\substack{q \bmod \SL_2(\bZ) \\ \Disc(q) = D}} \frac{1}{|\Stab_{\SL_2(\bZ)}(q)|} \sum_{(m,n)\ne 0 \in \bZ^2} q(m, n)^{-s} 
\]
where $q(m, n)$ denotes the value of the quadratic form $q$ at $(m, n ) \in \bZ^2$.  (That is, $q(m,n) = a m^2 + bmn + cn^2$ if $q = ax^2 + bxy + cy^2$.)
\end{proposition}
\begin{proof}
Consider the following  action of $\SL_2(\bZ)$ on pairs $(q, v)$, where $q = ax^2 + bxy + cy^2$ is an integral quadratic form and $v =(m,n) \in \bZ^2$ : 
\[
g(q, v) = (q \circ g^{-1}, g v).
\]
That is, $\SL_2(\bZ)$ is acting on the $\bZ^2$ component in the standard way.  The action on the binary quadratic form is not quite the usual one, which would be $g(v) = q \circ g^T$, but the orbits are the same. \\

This action has the property that it preserves the value $q(v)$, so we may define the sum
\[
\sum_{\substack{(q, v) \bmod \SL_2(\bZ) \\ \Disc (q) = D \\ v \ne (0,0) \in \bZ^2}} q(v)^{-s},
\]
which as we will see later converges absolutely for $s > 1$.  We first decompose this sum according to the $\SL_2$-orbit of $q$:
\begin{equation}\label{eq:decomp by q}
\sum_{\substack{(q, v) \bmod \SL_2(\bZ) \\ \Disc (q) = D \\ v \ne (0,0) \in \bZ^2}} q(v)^{-s} = \sum_{\substack{q \bmod \SL_2(\bZ) \\ \Disc(q) = D}} \sum_{\substack{v \ne 0 \bmod \Stab_{SL_2(\bZ)}(q)}} q(v)^{-s} = \sum_{\substack{q \bmod \SL_2(\bZ) \\ \Disc(q) = D}} \frac{1}{|\Stab_{\SL_2(\bZ)}(q)|} \sum_{v \ne 0 \in \bZ^2} q(v)^{-s}.
\end{equation}
To get the last equality above, we note that $\Stab_{\SL_2(\bZ)}(q)$ is finite and its nontrivial elements fix no nonzero element of $\bZ^2$.\\

Now we decompose our sum according to the $\SL_2(\bZ)$ orbit of $v$.  Note that every nonzero $\SL_2(\bZ)$ orbit on $\bZ^2$ contains a unique element of the form $(m, 0)$ with $m$ a positive integer.  Furthermore, the stabilizer in $\SL_2(\bZ)$ of $(m, 0)$ is exactly the group $U$ of upper-triangular matrices $\left\{ \pm \left(\begin{smallmatrix} 1 & *\\ 0 & 1 \end{smallmatrix}\right) \right\}$, whence
\begin{equation}
\sum_{\substack{(q, v) \bmod \SL_2(\bZ) \\ \Disc (q) = D \\ v \ne (0,0) \in \bZ^2}} q(v)^{-s} = \sum_{m \ge 1} \sum_{q \bmod U} q(m, 0)^{-s}.
\end{equation}
Now we write out $q = ax^2 + bxy + cy^2$, so that $q(m, 0) = m^2 a$.  We observe that the action of $U$ on binary quadratic forms preserves the coefficient $a \ge 1$, while it sends $b \mapsto b + 2 k a$ for arbitrary $k \in \bZ$.  Meanwhile, the coefficient $c = (b^2 - D)/(4a)$ is determined by $a$ and $b$, and is an integer exactly when $b ^2 \equiv D \pmod {4a}$. \\

Hence
\begin{equation}
\sum_{q \bmod U} q(m, 0)^{-s} = \sum_{a \ge 1} \sum_{\substack{b \bmod 2a \\ b^2 \equiv D \pmod {4a}}} (m^2a)^{-s} = m^{-2s}\sum_{a \ge 1} \#\{b \bmod 2a : b^2 \equiv D \pmod {4a} \} a^{-s},
\end{equation}
and plugging into the previous equation gives
\begin{equation} \label{eq:decomp by v}
\sum_{\substack{(q, v) \bmod \SL_2(\bZ) \\ \Disc (q) = D \\ v \ne (0,0) \in \bZ^2}} q(v)^{-s} = \sum_{m \ge 1} m^{-2s}\sum_{a \ge 1} \#\{b \bmod 2a : b^2 \equiv D \pmod {4a} \} a^{-s}.
\end{equation}
Comparing equations \ref{eq:decomp by q} and \ref{eq:decomp by v} gives the desired formula.
\end{proof}

Define \[L_D(s) = \zeta_D(s) / \zeta(s).\] The function $L_D(s)$ is sometimes called a ``Zagier $L$-series''  as Zagier \cite{Zagier} defined the series $L_D(s)$ and showed that it analytically continues to an entire function.  For our purposes, however, we will only need that $L_D(s)$ can be analytically continued to a neighborhood of $s =1$.  This can be checked by computing the local factors directly. For further details, see Section~\ref{random euler}. Note that at all but finitely many primes they are the same as those in the $L$-series $L(1, \chi_D)$, which converges (non-absolutely) for $\mathrm{Re}(s) > 0$.\\

The following form of the analytic class number formula which relates values of $L_D$ to Hurwitz class numbers can be found in the literature, see, e.g. (\cite{Zagier}, Proposition 3, (iv)).

\begin{proposition}\label{Hurwitz CNF}
For any negative discriminant $D$, not necessarily squarefree, we have
\[
\Hur(D) = \frac{1}{\pi} L_D(1) |D|^{1/2}
\]
\end{proposition}

We will give a self-contained proof here, in the spirit of the proof of the usual class number formula, which was originally proved by Dirichlet using binary quadratic forms.  This involves evaluating the residue of the pole of $\zeta_D(s)$ at $s = 1$ in two different ways, using Proposition~\ref{zeta formula}. Alternatively, this result can be deduced from Lemma~\ref{Hurwitz and Dirichlet L} by manipulating $L$-series. \\

The key tool in our computation will be the following lemma.

\begin{lemma}\label{quadratic pole lemma}
Let $q$ be a positive definite quadratic form with real coefficients.  Then
\[
\sum_{v \ne 0 \in \bZ^2} q(v)^{-s} = 2 \pi|\Disc(q)|^{-1/2}(s-1)^{-1} + O(1)
\]
as $s \to 1^+$.
\end{lemma}

\begin{proof}
    The left hand side converges absolutely for $s > 1$, so we can rearrange it as a Dirichlet series $\sum_{k \ge 1} a_k k^{-s}$, where 
    \[
    a_k = \# \{v \in \bZ^2 \setminus 0 \mid q(v) = k\}.
    \]

The partial sums 
\[
\sum_{k \le N} a_k = \# \{v \in \bZ^2 \setminus 0 \mid q(v) \le N\} = \# \{v \in \bZ^2 \mid q(v) \le N\} - 1
\]
can be estimated using the geometry of numbers as follows.  The region $ \{v \in \bR^2 \mid q(v) \le N\}$ has area $2 \pi N |\Disc(q)|^{-1/2}$, and its projections on the axes have length $O(\sqrt{N})$, so it contains $4 \pi N /|\Disc(q)| + O(\sqrt{N})$ lattice points. In this proof, the constants implicit in the big $O$ notation are all absolute constants which are independent of the choice of quadratic form $q$ or the values of any of the variables. \\

It follows that
\[
\sum_{k \le N} a_k = 2\pi |\Disc(q)| ^{-1/2} N + O(N^{1/2}).
\]
Next we decompose our Dirichlet series as 
\[
\sum_{k \ge 1} a_k k^{-s} = 2\pi |\Disc(q)| ^{-1/2} \zeta(s) + \sum_{k \ge 1} b_k k^{-s},
\]
where $b_k = a_k - 2\pi |\Disc(q)| ^{-1/2}$, and
\[
\sum_{k \le N} b_k = O(N^{1/2}).
\]
We apply partial summation to $\sum_{k \ge 1} b_k k^{-s}$, and obtain that, for any $s > 1$
\begin{equation*}
\begin{split}
\sum_{k \ge 1} b_k k^{-s}&= \lim_{N \to \infty} b_N N^{-s} - \int_1^{\infty} \left( \sum_{1 \le k \le t} b_k \right) s t^{-s-1} dt \\
&= \lim_{N \to \infty} O(N^{1/2 - s}) + s \int_1^{\infty} O(t^{-s-1/2}) dt
\end{split}
\end{equation*}
which is bounded as $s \to 1^+$. \\

We also have the Laurent expansion for the Riemann zeta function around $s =1$: $\zeta(s) = (s-1)^{-1} + O(1)$ as $s \to 1^+$,
and so we conclude
\[
\sum_{v \ne 0 \in \bZ^2} q(v)^{-s} = 2\pi |\Disc(q)| ^{-1/2} \zeta(s) + \sum_{k \ge 1} b_k k^{-s} = 2\pi |\Disc(q)| ^{-1/2}(s-1)^{-1} + O(1) 
\]
as $s \to 1^+$.
\end{proof}
\begin{remark}
The left-hand side of Lemma~\ref{quadratic pole lemma} is, up to a factor of 2, the Epstein zeta function of the quadratic form $q$, so this result also follows from known analytic properties of the Epstein zeta function such as the Kronecker limit formula.
\end{remark}

We are now ready to prove the class number formula for Hurwitz class numbers.
\begin{proof}[Proof of Proposition~\ref{Hurwitz CNF}]
   
    We will prove this by looking at the behavior of $\zeta_D(s)$ as $s \to 1^+$. On the one hand,
\begin{equation}
\begin{split}
\zeta_D(s) &= L_D(s) \zeta(s)\\
&= (L_D(1) + O_D(s-1)) ((s-1)^{-1}+ O(1))\\
&= L_D(1) (s-1)^{-1} + O_D(1)
\end{split}
\end{equation}
since $L_D(s)$ has analytic continuation to a neighborhood of $s =1$. (Here the subscript $D$ in the big $O$ notation indicates that the implicit constant depends on $D$.) On the other hand, applying Lemma~\ref{quadratic pole lemma} to Proposition~\ref{zeta formula} gives 
\begin{equation}
\begin{split}
\zeta_D(s) &= \sum_{\substack{q \bmod \SL_2(\bZ) \\ \Disc(q) = D}} \frac{1}{|\Stab_{\SL_2(\bZ)}(q)|} \sum_{(m,n)\ne 0 \in \bZ^2} q(m, n)^{-s} \\ &=
\sum_{\substack{q \bmod \SL_2(\bZ) \\ \Disc(q) = D}} \frac{1}{|\Stab_{\SL_2(\bZ)}(q)|}  \left( 2 \pi  |D|^{-1/2}(s-1)^{-1} +  O_D(1)\right) \\
& = \pi |D|^{-1/2} (s-1)^{-1} \left(\sum_{\substack{q \bmod \SL_2(\bZ) \\ \Disc(q) = D}} \frac{2}{|\Stab_{\SL_2(\bZ)}(q)|} + O_D(s-1)\right) \\
&= \pi |D|^{-1/2}\Hur(D) (s-1)^{-1} + O_D(1).
\end{split}
\end{equation}
These imply 
\begin{equation} \Hur(D) = \frac{1}{\pi} |D|^{1/2} L_D(1) + O_D(s-1),
\end{equation}
and taking $s \rightarrow 1^+$ gives the result. \end{proof}

\subsection{The Random Euler Product model}\label{random euler}

By the Chinese remainder theorem, $\zeta_D(s)$ factorizes as a product of local factors: $\zeta_D(s) = \prod_{p} \zeta_{D, p}(s)$, and likewise we can factor $L_D(s) = \prod_{p} L_{D, p}(s)$, where $L_{D, p}(s) = (1-p^{-s}) \zeta_{D, p}(s)$. \\

If $p^2\nmid D$ then $L_{D, p}(s) = (1- \chi_{D}(p)p^{-s})^{-1}$; in particular in this case, for fixed $p$ and varying $D$, $L_{D, p}(s)$ only depends on $D$ mod $p$.  The more general case is messier and depends both on the largest $a$ such that $p^{2a} \mid D$ and on the value of $\chi_D(p^{-2a}D)$.  For the case where $p$ is odd, the values are summarized in the table below, where the last part indicates the probability over all possible values of the discriminant $D$, namely integers which are $0$ or $1$ mod $4$.

\begin{center}
\begin{tabular}{ |c|c|c|c| } 
\hline
 & $\zeta_{D, p}(s)$& $L_{D, p}(s)$ & probability over all $D \equiv 0, 1 \pmod 4$ \\
\hline
 $\chi_D(p)  = 1$  & $(1- p^{-s})^{-2}$ & $(1-p^{-s})^{-1}$  & $\frac{p-1}{2p}$  \\ 
 \hline
 $\chi_D(p) = -1$ & $(1-p^{-2s})^{-1}$ & $(1+p^{-s})^{-1}$& $\frac{p-1}{2p}$ \\
 \hline
 $\chi_D(p) = 0$  and  $p^2\nmid D$  & $(1-p^{-s})^{-1}$ & 1 & $\frac{p-1}{p^2}$ \\
 
 \hline
 $D = p^{2a} D', \chi_{D'}(p) = 1$ & $*$ & $*$ & $\frac{p-1}{2p^{2a+1}}$ \\
 \hline
 $D = p^{2a} D', \chi_{D'}(p) = - 1$ & $*$ & $*$ & $\frac{p-1}{2p^{2a+1}}$ \\
 \hline
 $D = p^{2a} D', \chi_{D'}(p) = 0, p^2\nmid D$ &  $*$ & $*$ & $\frac{p-1}{p^{2a+2}}$ \\
\hline
\end{tabular}
\end{center}
Here the $\zeta$ and $L$-factors in the second half of the table are explicitly computable expressions depending on $a$; however we do not write them out here as we will not need their precise values for the computation in this paper. However, they would be necessary to get higher moments. \\

Note that for any fixed modulus $n$, the average over all $D \equiv 0, 1 \pmod 4$ of $\# \{a \bmod 2n : a^2 \equiv D \pmod {4n} \}$ is $1$ since the number of possible values of $a$ is the same as the number of possible values of $D$. \\

It follows that, heuristically, the average of $\zeta_D(s)$ over all $D \equiv 0, 1 \pmod 4$ should be $\zeta(2s)\zeta(s)$, and the average of $L_{D}(s)$ over all $D \equiv 0, 1 \pmod 4$ should be $\zeta(2s)$. We note that if one sets $s = 1$, this should agree with the known asymptotic for the average of Hurwitz class numbers. \\ 

We also get the corresponding local statements: for fixed odd $p$, the average over all $D \equiv 0, 1 \pmod 4$ of $\zeta_{D, p}(s)$ is heuristically $(1- p^{-2s})^{-1}(1-p^{-s})^{-1}$, and the  average of $L_{D, p}(s)$ over all $D \equiv 0, 1 \pmod 4$ is heuristically $(1- p^{-2s})^{-1}$. \\

Note that $(1- p^{-2s})^{-1}$ is the average of the $L$-factors $(1-p^{-s})^{-1}$ and $(1+p^{-s})^{-1}$ in the first two rows of the table.  This explains why restricting to $D$ prime, where only the first two rows of the table occur, does not affect the average size of Hurwitz class numbers.  As a consequence, the average of $L_{D, p}(s)$ over all $D$ with $p \mid D$ (the weighted average of the last 4 rows of the table) must also be $(1-p^{2s})^{-1}$. \\

However, the restriction $D = 1-4p'$, $p'$ prime, that we are interested in does affect the heuristic average, since for any prime modulus $p$, the congruence class $1 \pmod p$ now appears with density 0, lowering the average value of each local factor.  The new average value of the local $L$-factor $L_{D, p}$  can be computed from the updated table:

\begin{center}
\begin{tabular}{ |c|c|c|c| } 
\hline
 & $\zeta_{D, p}(s)$& $L_{D, p}(s)$ & probability over  $D = 1-4p'$ \\
\hline
 $\chi_D(p)  = 1$  & $(1- p^{-s})^{-2}$ & $(1-p^{-s})^{-1}$  & $\frac{p-3}{2(p-1)}$  \\ 
 \hline
 $\chi_D(p) = -1$ & $(1-p^{-2s})^{-1}$ & $(1+p^{-s})^{-1}$& $\frac{1}{2}$ \\
 \hline
cases with $p \mid D$ (averaged) & $(1- p^{-2s})^{-1}(1-p^{-s})^{-1}$ & $(1-p^{-2s})^{-1}$ & $\frac{1}{p-1}$\\ 
\hline
\end{tabular}
\end{center}
as
\begin{equation} 
\begin{split}
\mathbb{E}(L_{p,D} \mid D = 1-4p') &=
 (1-p^{-2s})^{-1} \cdot \left(\frac{p-3}{2(p-1)} \cdot (1+p^{-s}) + \frac{1}{2} \cdot (1-  p^{-s})+ \frac{1}{p-1} \cdot 1 \right) \\
&= (1-p^{-2s})^{-1} \left(1 - \frac{p^{-s}}{p-1} \right).
\end{split}
\end{equation}
Setting $s = 1$, we get that, for $p$ an odd prime,
\begin{equation}\label{eq:compare}
\mathbb{E}(L_{p,D}(s) \mid D = 1-4p') = \mathbb{E}(L_{p, D}(1) \mid D \text{ any discriminant} ) \cdot \left(1- \frac{1}{p(p-1)} \right)
\end{equation}

The case $p = 2$ also needs to be checked separately.  We still have $\mathbb{E}(L_{2, D}(1) \mid D \text{ any discriminant} ) = (1-2^{-2s})^{-1}$ by the same argument as before. 
 On the other hand, for odd primes $p'$, $D = 1-4p'$ is $5 \mod 8$, so $\chi_2(D) = -1$ and $L_{2, D}(s) = (1+2^{-s})^{-1}$.  Setting $s = 1$, we obtain $\mathbb{E}(L_{2,D}(1) \mid D = 1-4p') = 2/3$ and $\mathbb{E}(L_{p, D}(1) \mid D \text{ any discriminant} ) = 4/3$, verifying \eqref{eq:compare} for $p = 2$ also. \\


Therefore, heuristically treating the local zeta factors as independent random variables, we expect that restricting to discriminants $D = 1 - 4p'$ with $p$ prime should multiply the average class number by \[\prod_{\ell \ge 2 \text{ prime}} \left(1 - \frac{1}{\ell(\ell-1)} \right),\] which coincidentally is exactly the definition of Artin's constant.  By the prime number theorem, this is consistent with our result Theorem~\ref{MT} and its refined version Theorem~\ref{main result CNF}.


\section{Preliminary lemmas}

\subsection{An unrestricted bound for bilinear character sums}

We begin with the statement of the following lemma, which originates from \cite{FK}: 

\begin{lemma}[Bound on bilinear character sums] \label{double} Let $\{\alpha_n\}, \{\beta_m\}$ be two sequences of complex numbers with each term having absolute value bounded by $1$. Let $M,N$ be positive real numbers. Then  we have
\[\sum_{m \leq M} \sum_{n \leq N} \alpha_m \beta_n \mu^2(2m) \mu^2(2n) \Kronecker{m}{n} \] 
\[ \ll MN \min \left\{ \left(M^{-1/2} + (N/M)^{-1/2}  \right), \left(N^{-1/2} + (M/N)^{-1/2} \right) \right\}.  \]
\end{lemma}

Lemma \ref{double} essentially follows from the Polya-Vinogradov inequality. There are several instances in this paper where Lemma \ref{double} is needed.  However, the sharpest form of Lemma \ref{double} requires that the sum is supported on \emph{odd squarefree numbers}. We observe here that applying a squarefree sieve allows us to remove this restriction at the cost of a logarithmic factor, which is often an acceptable loss. We thus obtain the following variant of Lemma \ref{double}: 

\begin{lemma}[Unrestricted bound on bilinear character sums] \label{unrestricted double oscillation} Let $\{\alpha_n\}, \{\beta_m\}$ be two sequences of complex numbers with each term having absolute value bounded by $1$. Let $M,N$ be positive real numbers. Then  we have
\[\sum_{m \leq M} \sum_{n \leq N} \alpha_m \beta_n \Kronecker{m}{n}\] 
\[ \ll MN (\log M + \log N) \min \left\{ \left(M^{-1/2} + (N/M)^{-1/2}  \right), \left(N^{-1/2} + (M/N)^{-1/2} \right) \right\}.  \]
\end{lemma}

Note that Lemma~\ref{double} can be restated as the following:
\begin{equation}\label{double bound}
\sum_{m \leq M} \sum_{n \leq N} \alpha_m \beta_n   \Kronecker{m}{n} 
 \ll MN \min \left\{ \left(M^{-1/2} + (N/M)^{-1/2}  \right), \left(N^{-1/2} + (M/N)^{-1/2} \right) \right\}.  
\end{equation}
where the sequences $\{\alpha_n\}, \{\beta_m\}$ are supported on odd squarefree numbers.  Comparing with Lemma~\ref{unrestricted double oscillation} our task is to remove the two restrictions to odd and squarefree numbers at the cost of worsening the bound.  The restriction to odd numbers is easy and only affects the constant in the bound, while the restriction to squarefree numbers is the main part of the proof and contributes the extra log factor.

\begin{proof}[Proof of Lemma \ref{unrestricted double oscillation}]
First we deal with the restriction to odd numbers.  As every even squarefree number is twice an odd squarefree number, we can decompose the double sum over all $m$, $n$ in \eqref{double bound} as a sum of four sums depending on the parity of $m$ and $n$.  Since any even squarefree number is twice an odd squarefree number, each of these can be written as sums over odd squarefree numbers, and thus can be bounded using Lemma~\ref{double}.  We conclude that that \eqref{double bound} holds for all sequences $\{\alpha_n\}, \{\beta_m\}$  supported on (not necessarily odd) squarefree numbers, although with a larger implicit constant.

The rest of the proof is devoted to removing the squarefree restriction
By symmetry, it is enough to show that our left hand side $\Xi(M, N) = \sum_{m \leq M} \sum_{n \leq N} \alpha_m \beta_n \Kronecker{m}{n}$ satisfies $\Xi(M, N) \ll ( M^{1/2} N + M^{3/2} N^{1/2} ) (\log M + \log N)$. \\

We then decompose $\Xi(M, N)$ as 
\[\Xi(M, N) = \sum_{k \le \sqrt{M}} \sum_{\ell \le \sqrt{N}} \Xi_{k, \ell}(M, N)\]
where 
\begin{align*}
\Xi_{k, \ell}(M, N) &= \sum_{m' \le k^{-2}M \text{ squarefree}} \sum_{n' \le \ell^{-2}M \text{ squarefree}} \alpha_{k^2 m'} \beta_{\ell^2 n'} \Kronecker{k^2 m'}{\ell^2 n'}\\
&=\sum_{m' \le k^{-2}M \text{ squarefree}} \sum_{n' \le \ell^{-2}M \text{ squarefree}} \alpha_{k^2 m'} \beta_{\ell^2 n'} \Kronecker{k^2}{\ell^2} \Kronecker{k^2}{n'} \Kronecker{m'}{\ell^2} \Kronecker{m'}{n'} \\
& = \Kronecker{k^2}{\ell^2} \sum_{m' \le k^{-2}M \text{ squarefree}} \sum_{n' \le \ell^{-2}M \text{ squarefree}} \alpha'_{m'} \beta'_{n'} \Kronecker{m'}{n'}.
\end{align*}
is the total contribution from all pairs $(m, n)$ of the form $(k^2 m', \ell^2 n')$ with $m'$ and $n'$ squarefree, where $\alpha'_{m'} = \Kronecker{m'}{\ell^2} \alpha_{k^2 m'} $ and $\beta'_{m'} = \Kronecker{m'}{\ell^2} \beta_{\ell^2 n'}$.
Since $m'$ and $n'$ are now restricted to be squarefree, we can now apply \eqref{double bound} to obtain
\[
\Xi_{k, \ell}(M, N) \ll (k^{-1} \ell^{-2} M^{1/2} N + k^{-3} \ell^{-1} M^{3/2} N^{1/2}).
\]
Summing over $k, \ell$ gives the bound
\[
\Xi(M, N) \ll  \sum_{k \le \sqrt{M}} \sum_{\ell \le \sqrt{N}}  (k^{-1} \ell^{-2} M^{1/2} N + k^{-3} \ell^{-1} M^{3/2} N^{1/2}) \asymp \log(M)M^{1/2} N + \log(N) M^{3/2} N^{1/2},
\]
which is of the size needed.
\end{proof}

\subsection{A Siegel--Walfisz-type lemma} 

As is well-known by now, in problems involving summations over Kronecker symbols where both arguments are variable, it is often necessary to consider the sum over just one variable. In such cases Lemma \ref{double} does not apply, and so we will require the following result, which is derived from the Siegel--Walfisz theorem. 



\begin{lemma}[Siegel--Walfisz] \label{S-W new}
Let $A$ be a positive real number and $(\alpha_k)$, $k \ge 1$ be a sequence of complex numbers, periodic with period  $q \geq 1$, such that $|\alpha_k| \le 1$ for all $k$.
\[
\sum_{p \le X} \alpha_p = \frac{\sum_{a \in \bZ/q\bZ^*} \alpha_a}{\phi(q)} \Li(X) + O_A \left(q  X (\log X)^{-A} \right)
\]
uniformly in $X$ and $q$.
\end{lemma}

\begin{proof}
We have
\[
\sum_{p \le X} \alpha_p = \sum_{a \in  \bZ/q\bZ^*} \alpha_a \pi(X; a, q)
\]
where $\pi(X; a, q)$ denotes the number of primes $p \le X$ congruent to $a$ modulo $q$.\\

Now the Siegel--Walfisz theorem \cite[Corollary 5.29]{IK} says that 
\[
\pi(X; a, q) = \frac{1}{\phi(q)} \Li(X) + O_A(X (\log X)^{-A} ).
\]
  Taking the weighted sum of this over all congruence classes modulo $q$, we get
\[
\sum_{p \le X} \alpha_p = \frac{\sum_{a \in \bZ/q\bZ^*} \alpha_a}{\phi(q)} \Li(X) + \phi(q) O_A \left( X (\log X)^{-A} \right) = O_A \left(q X (\log X)^{-A} \right)
\]
since $\phi(q) \le q$.\end{proof}

We note that the statement of Lemma \ref{S-W new} implies that we can absorb the dependence on $q$ in the error term, provided that $q \leq (\log X)^{A^\prime}$ for some positive number $A^\prime$, and then adjusting $A$ accordingly. This is the usual way that Siegel--Walfisz is stated. \\

In our proof, we will need to apply this to the sequence  $\alpha_k = \Kronecker{(1-4k)/d^2}{n}$, for arbitrary positive integers $n$, $d$, which is periodic with period $d^2n$.  For this, we will need to calculate the following quantity, which will appear in the coefficient of the main term: for integers $n$ and $d$ with $d$ odd, define 
\begin{equation}\label{Define A(n,d)}
\A(n, d) = \frac{\phi(d^2)}{\phi(n d^2)}\sum_{\substack{a \in (\bZ/d^2 n\bZ)^\ast \\ d^2 \mid 1-4a} } \Kronecker{(1-4a)/d^2}{n}.
\end{equation}

\begin{lemma}\label{sum mod d^2m}   
For fixed $d$, $\A(n, d)$ is multiplicative as a function of $n$. Furthermore, we have the following identity of Dirichlet series:

\begin{equation}\label{Dirichlet}
 \sum_{n} \A(n, d) n^{-s}
= (1-2^{-s}) \zeta(2s) \prod_{\ell \nmid d \text{ odd}} \left( 1 - \left(\frac{1}{\ell-1}\right) ( \ell^{-s} + \ell^{-2s}) \right) \prod_{\ell \mid d} (1- \ell^{-2s-1}),
\end{equation}
which converges absolutely for $s > 1/2$.
\end{lemma}

\begin{proof}


We change variables to $b = (1-4a)/d^2$, which is well-defined as an element of $\bZ / 4n \bZ$.  We then restrict ourselves to a sum over classes $b \in \bZ / 4n \bZ$ such that $d^2 b \equiv 1 \pmod 4$ and $(1-d^2 b)/4$ is relatively prime to $ d^2 n$.  Note that $(1-d^2 b)/4$ is relatively prime to $d^2$, so it's sufficient to require that $(1-d^2 b)/4$ is a unit modulo $n$.\\

Define  $\mathcal{S}(n, d) :=  \{ b \in \bZ / 4n \bZ \mid d^2 b \equiv 1 \pmod 4, (1-d^2 b)/4 \text { is a unit modulo }n\} $, so that 
\begin{equation}\label{formula}
    \A(n, d) = \frac{\phi(d^2)} {\phi(nd^2)} \sum_{b \in \mathcal{S}(n, d)} \Kronecker{b}{n}.
\end{equation}
We note that 
\[\frac{\phi(d^2)} {\phi(nd^2)} = \left(n \cdot \prod_{\substack{\ell \mid n\\ \ell \nmid d}} \left(1 - \frac{1}{\ell} \right) \right)^{-1}\] 
is a multiplicative function of $n$, so it suffices to show that $\sum_{b \in \mathcal{S}(n, d)} \Kronecker{b}{n}$ is multiplicative.\\

Now suppose $n = n_1 n_2$ with $\gcd(n_1, n_2) = 1$ .
By the Chinese remainder theorem, there is a natural bijection $\mathcal{S}(n, d)\to \mathcal{S}(n_1, d) \times \mathcal{S}{(n_2, d)}$, given by $b \bmod 4n \mapsto (b \bmod 4n_1, b \bmod 4n_2)$.  Furthermore, the value of the Kronecker symbol $\Kronecker{b}{n_i}$ depends only on the value of $b$ modulo $4 n_i$, and therefore
\begin{equation}
\sum_{b \in \mathcal{S}(n, d)} \Kronecker{b}{n} = \left( \sum_{b \in \mathcal{S}(n_1, d)} \Kronecker{b_1}{n} \right) \cdot \left( \sum_{b \in \mathcal{S}(n_2, d)} \Kronecker{b_2}{n} \right).
\end{equation}
This completes our proof of multiplicativity.\\

We now check the identity of Dirichlet series. By the multiplicativity proved above, it's enough to check for each prime $p$ that $\sum_k \A(\ell^k, d) (\ell^k)^{-s}$ agrees with the Euler factor at $p$ in the right hand side of \eqref{Dirichlet}.
We break up into cases, and use equation \eqref{formula} to compute $\A(\ell^k, d)$ in each of these cases. \\

{\bf Case $\ell = 2$:} 
Here
\[
\A(2^k, d) = \frac{\phi(d^2)}{\phi(2^k d^2)} \sum_{b \in \mathcal{S}(2^k, d)} \Kronecker{b}{2^k}.
\]
For $k \ge 1$, the first factor is $2^{1-k}$ by multiplicativity of $\phi$.
We compute 
\[\mathcal{S}(2^k, d) = \{b \in \bZ/(2^{k+2} \bZ) \mid  b \equiv 5 \pmod 8\}\] 
for any odd $d$.   This has cardinality $2^{k-1}$ for $k \ge 1$,
and $\Kronecker{b}{2^k} = (-1)^k$ for any $b \in \mathcal{S}(2^k, d)$, so $\sum_{b \in \mathcal{S}(2^k, d)} \Kronecker{b}{2^k} = (-1)^k 2^{k-1}$ for $k \ge 1$, and
\[
\A(2^k, d) = \frac{\phi(d^2)} {\phi(2^k d^2)} \sum_{b \in \mathcal{S}(2^k, d)} \Kronecker{b}{2^k} = 2^{1-k} \left( (-1)^k 2^{k-1}\right) = (-1)^k
\]
for $k \ge 1$.  We also have $\A(1, d) = 1$, so
\[
\sum_{k \ge 0} \A(2^k, d) (2^k)^{-s} = \sum_{k \ge 0} (-1)^k (2^k)^{-s} = \frac{1- 2^{-s}}{1-2^{-2s}},
\]
agreeing with the Euler factor at 2  in \eqref{formula}.\\

{\bf Case $\ell \mid d$:}
Here
\[
\A(\ell^k, d) = \frac{\phi(d^2)}{\phi(\ell^k d^2)} \sum_{b \in \mathcal{S}(\ell^k, d)} \Kronecker{b}{p^k}
\]
The first factor is $\phi(d^2)/\phi(\ell^k d^2) = \ell^{-k}$ because $\ell \mid d$.  To compute the sum, note that $(1-d^2 b)/4$ is always a unit mod $\ell \mid d$. It follows that 
\[\mathcal{S}(\ell^k, d) = \{ b \in \bZ / 4\ell^k \bZ \mid b\equiv 1 \pmod 4\},\] 
and thus for $k \ge 1$
\[
\sum_{b \in \mathcal{S}(\ell^k, d)} \Kronecker{b}{\ell^k} = \ell^{k-1} \sum_{b \in (\bZ / \ell \bZ)} \Kronecker{b}{\ell}^k= \begin{cases} (\ell-1) \ell^{k-1} & k \text{ even} \\ 0 & k \text{ odd}.
\end{cases}
\]
We conclude that
\[
\A(\ell^k, d) = \begin{cases} 1 & k = 0\\
        1- \ell^{-1} & k > 0 \text{ even} \\
        0 & k \text{ odd} 
    \end{cases}
\]
Hence 
\[
\sum_k \A(\ell^k, d)(\ell^k)^{-s} = \sum_{j \ge 0} \ell^{-2j} - \ell \sum_{j \ge 1} \ell^{-2j} = \zeta(2s) - \ell^{-1-2s}\zeta(2s) = \zeta(2s)(1- \ell^{-1-2s})
\]
as desired.\\

{\bf Case $\ell \nmid 2d$:}
Again we evaluate
\[
\A(\ell^k, d) = \frac{\phi(d^2)}{\phi(\ell^k d^2)} \sum_{b \in \mathcal{S}(\ell^k, d)} \Kronecker{b}{\ell^k}
\]
In this case, for $k \ge 1$ the first term is $\frac{1}{\phi(\ell^k)} = (\ell-1)^{-1}\ell^{1-k}$ by multiplicativity of $\phi$.  To calculate the sum, observe that $\mathcal{S}(\ell^k, d) = \{ b \in \bZ / 4\ell^k \bZ \mid b\equiv 1 \pmod 4, b \not \equiv d^{-2} \pmod \ell\}$, and therefore
\[
\sum_{b \in \mathcal{S}(\ell^k, d)} \Kronecker{b}{\ell^k} = \ell^{k -1} \sum_{\substack{b \in (\bZ / \ell \bZ) \\ b \not \equiv d^{-2} \pmod \ell}} \Kronecker{b}{\ell}^k = \begin{cases}
    (\ell-2) \ell^{k-1} & k \text{ even}\\
    -\ell^{k-1} & k \text{ odd}.
\end{cases}
\]
We obtain
\[
\A(\ell^k, d) = \begin{cases} 1 & k = 0\\
\frac{\ell-2}{\ell-1}  & k > 0 \text{ even} \\
\frac{-1}{\ell-1} & k \text{ odd}
\end{cases}
\]
and hence
\begin{equation*}
\begin{split}
\sum_k \A(\ell^k, d)(\ell^k)^{-s} & = \frac{1}{\ell-1} +  \sum_{j \ge 0} \left( \frac{\ell-2}{\ell-1} (\ell^{2j})^{-s} + \frac{-1}{\ell-1} (\ell^{2j+1})^{-s} \right) \\ 
&= 
\frac{1}{\ell-1} + (1- \ell^{-2s})^{-1} \left (\frac{\ell-2} {\ell-1} - \frac{1}{\ell-1} \ell^{-s} \right) \\
&= ( 1 -\ell^{-2s})^{-1} \left(1 - \frac{1}{\ell-1} \right) (\ell^{-s} + \ell^{-2s})
\end{split}
\end{equation*}
as desired. Now, the formal Dirichlet series given by 
\[
\sum_{n} \A(n, d) n^{-s}
= (1-2^{-s}) \zeta(2s) \prod_{\ell \nmid d \text{ odd}} \left( 1 - \left(\frac{1}{\ell-1}\right) ( \ell^{-s} + \ell^{-2s}) \right) \prod_{\ell \mid d} (1- \ell^{-2s-1})
\]
converges absolutely for $s > 1/2$, as the factor $\zeta(2s)$ converges for $s > 1/2$ and 
\[\prod_{\ell \nmid d \text{ odd}} \left( 1 - \left(\frac{1}{\ell-1}\right) (\ell^{-s} + \ell^{-2s} \right) \] 
converges for $s > 0$.

\end{proof}







\section{Proof via the analytic class number formula}

We now give our proof, which is similar to the approaches of Nagoshi \cite{Nagoshi} for discriminant $1-4p$ and Friedlander-Iwaniec \cite{FI hyperbolic} for discriminant $-4p$.    We estimate the sum of class numbers via the analytic class number formula.\\

By Corollary~\ref{difference with Hur}, instead of estimating $\sum_{p \le X} H(1-4p)$, we can estimate \[Q(X) := \sum_{p \le X} \Hur(1-4p),\] as the difference will be absorbed into the error term.  
Using the analytic class number formula in the form Lemma~\ref{Hurwitz and Dirichlet L}, and interchanging the order of summation, we can write
\begin{equation}
\begin{split}
    Q(X) &= \sum_{p \le X} \sum_{d^2 \mid 1-4p} \frac{h((1-4p)/d^2)}{w((1-4p)/d^2)} = \sum_{p \le X} \sum_{d^2 \mid 1-4p} \frac{1}{\pi} L(1, \chi_{(1-4p)/d^2}) \left|\frac{4p-1}{d^2}\right|^{1/2} \\
    & = \sum_{d \ge 1 \text{ odd}} Q_d(X) 
\end{split}
\end{equation}
where
\begin{equation}
\begin{split}
Q_d(X)
:= \sum_{\substack{p \le X \\ d^2 \mid 1-4p}}\frac{1}{\pi} L(1, \chi_{(1-4p)/d^2}) \left|\frac{1-4p}{d^2}\right|^{1/2} \\
= \frac{1}{\pi d}\sum_{\substack{p \le X \\ d^2 \mid 1-4p}} \sum_{n \ge 1} \frac{(4p-1)^{1/2}}{n} \Kronecker{(1-4p)/d^2}{n}.
\end{split}
\end{equation}

It will be convenient to us to first bound the related quantity 

\[
T_d(X) = \sum_{\substack{p \le  X\\ d^2 \mid 1-4p}} L(1, \chi_{(1-4p)/d^2}) = \sum_{\substack{p \le  X\\ d^2 \mid 1-4p}} \sum_{n \ge 1} \Kronecker{(1-4p)/d^2}{n} \frac{1}{n}.
\]

We define a multiplicative function $c(d)$, which will feature in the leading terms. For $d$ odd we set
\begin{equation}\label{define c}
   c(d) :=  \frac{1}{d^3} \prod_{\ell \mid d} \frac{\ell^3 -1 }{\ell^3 - \ell^2 - \ell - 1}
\end{equation}
and for $d$ even we set $c(d) = 0$. \\

We can now state the main results to be proved.

\begin{lemma} \label{estimate $T_d$}
If $d < (\log X)^{\alpha}$ for some fixed $\alpha > 0$ and $X \ge 2$, then

 \[T_d(X) = \frac{\pi^2}{12} \cdot d c(d) \prod_{\ell \text{ odd}}\frac{\ell^3 - \ell^2 - \ell - 1}{\ell^3 - \ell^2} \Li(X) + O_{\alpha, A}\left(\frac{X}{(\log X)^A}\right),\]
 for  any real $A\ge 1$.
\end{lemma}

\begin{lemma} \label{estimate $Q_d$}
If $d < (\log X)^{\alpha}$ for some fixed $\alpha > 0$ and $X \ge 2$, then
 \[ Q_d(X) = \frac{\pi}{6} c(d) \prod_{\ell \text{ odd}}\frac{\ell^3 - \ell^2 - \ell - 1}{\ell^3 - \ell^2} \int_2^{X} \frac{t^{1/2}}{\log t} dt + O_{\alpha, A} \left(\frac{X^{3/2}}{(\log X)^A}\right)\]
for any real $A\ge 1$.
\end{lemma}

These two lemmas allow us to obtain expressions with an arbitrary power-log saving in the error term. This is in part due to the Siegel--Walfisz theorem. From these lemmas we obtain our main result: 

\begin{theorem}\label{main result CNF}
\[
Q(X) = 
\frac{\pi}{3} C_{\operatorname{Art}} \int_2^X \frac{t^{1/2}}{\log t} dt + O_A\left(\frac{X^{3/2}}{(\log X)^A}\right)
\]
for any real number $A>1$. It follows that
\[
\sum_{p \le X} H(1-4p) = \frac{\pi}{3} C_{\operatorname{Art}} \int_2^X \frac{t^{1/2}}{\log t} dt + O_A\left(\frac{X^{3/2}}{(\log X)^A}\right).
\]
\end{theorem}

We note that Theorem \ref{MT} follows from Theorem \ref{main result CNF}, since 

\begin{align} \label{MTIBP} \int_2^X \frac{t^{1/2}}{\log t} dt & = \left[\frac{2}{3} \frac{t^{3/2}}{\log t} \right]_2^X - \frac{2}{3} \int_2^X \frac{t^{3/2}}{t (\log t)^2}  dt \\
& = \frac{2}{3} \cdot \frac{X^{3/2}}{ \log X} + O \left(\frac{X^{3/2}}{(\log X)^2} \right). \notag
\end{align}

Lemma~\ref{estimate $T_d$} is the key technical result, which we will prove in Section~\ref{$T_d$ details}.  We now assume it and prove the other results. 

\begin{proof}[Proof of Lemma~\ref{estimate $Q_d$}]

We apply partial summation to 

\begin{align*}
  Q_d(X) &= \sum_{\substack{p \le X \\ d^2 \mid 1-4p}}\frac{1}{\pi} L(1, \chi_{(1-4p)/d^2}) \left|\frac{4p-1}{d^2}\right|^{1/2} 
  \\ &= \frac{1}{\pi d}  \left((4X-1)^{1/2} T_d(X)   - \sqrt{7}T_d(2) - \int_{2}^{\infty} (t - \tfrac{1}{4})^{-1/2} T_d(t) dt \right)
  \\ &= \frac{1}{\pi d} \left(2 X^{1/2} T_d(X) - \int_{2}^X t^{-1/2} T_d(t) dt + O\left(\frac{X^{1/2}}{\log X}\right) \right).
\end{align*}
where at the last step we are using Lemma~\ref{estimate $T_d$} in the weaker form $T_d(X) = O_{\alpha, A} \left( \frac{X}{\log X} \right)$.  Note that the last error term is a power-saving over $X$, which is more than enough for our purpose. \\


Now we estimate the integral; for convenience we write $f(d) = \frac{\pi^2}{12} \cdot d c(d) \prod_{\ell \text{ odd}}\frac{\ell^3 - \ell^2 - \ell - 1}{\ell^3 - \ell^2}$, so that Lemma~\ref{estimate $T_d$}, with the variable changed to $t$, reads
\[
T_d(t) = f(d) \Li(t) + O_{\alpha, A}\left(\frac{t}{(\log t)^A}\right).
\]
Multiplying this by $t^{-1/2}$ and integrating, we get
\begin{align} \label{intd}
    \int_{2}^X t^{-1/2} T_d(t) dt & = f(d) \int_2^X t^{-1/2} \Li(t) dt + f(d) \int_{2}^X O_{\alpha, A} \left(\frac{t^{-1/2}}{(\log t)^A}\right)dt \\
    & = f(d)\int_2^X t^{-1/2} \Li(t) dt +O_{\alpha, A} \left(\frac{X^{3/2}}{(\log X)^A}\right), \notag
\end{align}
where the last error term comes from integrating an integrand of size $O_{\alpha, A} \left(\frac{X^{1/2}}{(\log X)^A}\right)$ over an integral of length $X$. The implied constant can be made independent of $d$, on the observation that $f(d)$ is absolutely bounded. \\

We now evaluate the integral 
\[\int_2^X t^{-1/2} \Li(t) dt\]
using integration by parts. We obtain 
\begin{align*} \int_2^X t^{-1/2} \Li(t) dt & = \left[2 t^{1/2} \Li(t) \right]_2^X - \int_2^X \frac{2 t^{1/2}}{\log t}dt \\
& = 2X^{1/2} \Li(X) - \int_2^X \frac{2t^{1/2} dt}{\log t} + O(1).
\end{align*}
Combining with (\ref{intd}) we obtain 
\begin{align} Q_d(X) & = \frac{2 f(d)}{\pi d} \int_2^X \frac{t^{1/2}}{\log t} dt + O_{\alpha, A} \left(\frac{X^{3/2}}{(\log X)^A} \right) \\
& = \frac{\pi}{6} \cdot c(d) \prod_{\ell \text{ odd}} \frac{\ell^3 - \ell^2 - \ell - 1}{\ell^3 - \ell^2} \int_2^X \frac{t^{1/2} dt}{\log t} + O_{\alpha, A} \left(\frac{X^{3/2}}{(\log X)^A} \right)\notag
\end{align}
as desired. 


\end{proof}

We now sum over $d$ to obtain the asymptotic for $Q$:

\begin{proof}[Proof of Theorem~\ref{main result CNF}]

Using $L(1, \chi_D) = O(\log D)$, we have the trivial bound

\[
Q_d(X) \ll \sum_{\substack{m \le  X\\ d^2 \mid 1-4m}} \log X \frac{\sqrt{X}}{d} \asymp d^{-3} X^{3/2} \log X.  
\]
We cut off the sum $\sum_{d \in [1, X] \text{ odd}} Q_d(X)$ at $d = (\log X) ^{\alpha}$ for some real constant $\alpha$ whose value we will determine later. 
\begin{equation}
\begin{split}
Q(X) &= \sum_{\substack{1 \leq d \leq X \\ d \text{ odd}}} Q_d(X)\\
&=  \sum_{\substack{1 \leq d < (\log X)^\alpha \\ d \text{ odd}}} Q_d(X) + \sum_{d \ge (\log X)^\alpha} Q_d(X)\\
&\le  \sum_{\substack{1 \leq d < (\log X)^{\alpha} \\ d \text{ odd}}} Q_d(X) +  \sum_{d \ge (\log X)^\alpha} O(d^{-3} X^{3/2} \log X)\\
&
= \sum_{\substack{1 \leq d < (\log X)^{\alpha} \\ d \text{ odd}}} Q_d(X) + O(X^{3/2} (\log X)^{1-2 \alpha})
\end{split}
\end{equation}

We now apply Lemma~\ref{estimate $Q_d$} to the cut off sum to obtain 
\begin{equation}
\label{apply estimate $Q_d$}
  \sum_{\substack{1 \leq d < (\log X)^\alpha \\d \text{ odd}}} Q_d(X) =  \prod_{\ell \text{ odd}}\frac{\ell^3 - \ell^2 - \ell - 1}{\ell^3 - \ell^2} \left( \sum_{1 \leq d < (\log X)^{\alpha}} c(d) \right) \int_{2}^X \frac{t^{1/2}}{\log t} dt + \sum_{1 \leq d < (\log X)^\alpha}  O_{\alpha, C} \left(\frac{X^{3/2}}{(\log X)^A}\right).
\end{equation}
where we are allowed to choose the real parameter $A > 1$ freely, and  we make the convention $c(d) = 0$ for $d$ even.  The error term here is of size size $\log(X)^{\alpha} \cdot O_{A, \alpha} \left(\frac{X^{3/2}}{(\log X)^A}\right) = O_{A, \alpha}(X^{3/2} \log{X}^{\alpha - A})$.  Choosing $A$ in terms of $\alpha$ large enough so that $\alpha - A \le 1- 2 \alpha$, we get that

\begin{equation}
\label{put together Q}
Q(X)  =    \prod_{\ell \text{ odd}}\frac{\ell^3 - \ell^2 - \ell - 1}{\ell^3 - \ell^2} \left( \sum_{1 \leq d < (\log X)^\alpha} c(d) \right) \int_{2}^X \frac{t^{1/2}}{\log t} dt  + O_\alpha(X^{3/2} \log{X}^{1 -2 \alpha})
\end{equation}

We now need to evaluate the  sum $\sum_{d = 1}^{\infty} c(d)$, which we can expand as an Euler product:
\begin{equation*}
\begin{split}
    \sum_{d = 1}^{\infty} c(d) & = \sum_{d \ge 1 \text{ odd}}  d^{-3} \prod_{\ell \mid d} \frac{\ell^3 -1 }{\ell^3 - \ell^2 -\ell -1} \\
    & = \prod_{\ell \text{ odd}} 1+ \ell^{-3}(1 - \ell^{-3})^{-1} \frac{\ell^3 -1 }{\ell^3 - \ell^2 -\ell -1} \\
    &= \prod_{\ell \text{ odd}} \frac{\ell^3 - \ell^2 - \ell}{\ell^3 - \ell^2 - \ell -1}
\end{split}
\end{equation*}

We write 
\[\sum_{d=1}^\infty c(d) = \sum_{1 \le d < (\log X)^\alpha} c(d) + \sum_{d \ge (\log X)^\alpha} c(d).\]
From (\ref{define c}) and an application of Mertens' theorem, we see that 
\begin{equation} \label{trivcbd} c(d) \ll \frac{\log \log d}{d^3}.
\end{equation}

It follows that 
\begin{align*} \sum_{1 \leq d < (\log X)^{\alpha}} c(d) & = \sum_{d \geq 1} c(d) + O \left(\sum_{d \ge (\log X)^\alpha} \frac{\log \log d}{d^3} \right) \\
& = \prod_{\ell \text{ odd}} \frac{\ell^3 - \ell^2 - \ell}{\ell^3 - \ell^2 - \ell - 1} + O_\alpha \left(\frac{\log \log \log X}{(\log X)^{2 \alpha}} \right) \\ 
&= \prod_{\ell \text{ odd}} \frac{\ell^3 - \ell^2 - \ell}{\ell^3 - \ell^2 - \ell - 1} + O_{\alpha} \left( (\log X)^{-2\alpha + 1} \right).
\end{align*}

 Plugging this into \eqref{apply estimate $Q_d$}, we obtain

\begin{equation*}
\begin{split}
Q(X) &= \frac{\pi}{6} \prod_{\ell \text{ odd}} \left( \frac{\ell^3 - \ell^2 - \ell}{\ell^3 - \ell^2 - \ell -1} \cdot \frac{\ell^3 - \ell^2 - \ell - 1}{\ell^3 - \ell^2} \right) \int_2^X \frac{t^{1/2}}{\log t} dt+ O_\alpha(X^{3/2} \log(X)^{-2 \alpha} )+ O_{\alpha}(X^{3/2} \log(X)^{1-2\alpha})\\
& =  \frac{\pi}{6} \prod_{\ell \text{ odd}} \left(1 - \frac{1}{\ell^2 - \ell} \right) \int_2^X \frac{t^{1/2}}{\log t} dt+ O_{\alpha}(X^{3/2} \log(X)^{1-2\alpha})\\
& = \frac{\pi}{3} C_{\text{Art}} \int_2^X \frac{t^{1/2}}{\log t} dt + O_{\alpha}(X^{3/2} \log(X)^{1-2\alpha}).
\end{split}
\end{equation*} 
Since we can take $\alpha$ to be arbitrarily large, this gives a bound of the desired form.\end{proof}

\subsection{Proof of Lemma~\ref{estimate $T_d$}}\label{$T_d$ details}

We now need to estimate \[
T_d(X) = \sum_{\substack{p \le  X\\ d^2 \mid 1-4p}} L(1, \chi_{(1-4p)/d^2)}) = \sum_{\substack{p \le  X\\ d^2 \mid 1-4p}} \sum_{n \ge 1} \Kronecker{(1-4p)/d^2}{n} \frac{1}{n}.
\]
under the assumption that $d < (\log X)^{\alpha}$.\\

First we want to cut off the tails of the inner sums, so that we have something finite.  Define
\[
T_d(X, [B_1, B_2]) := \sum_{\substack{p \le  X\\ d^2 \mid 1-4p}} \sum_{B_1\le n \le B_2} \Kronecker{(1-4p)/d^2}{n} \frac{1}{n}.
\]

We cut off our sum at some real number $B$, decomposing
\[
T_d(X) = T_d(X, [1, B]) + T_d(X, [B, \infty]).
\]  

The tail $T_d(X, [B, \infty])$ is easy to bound.  By equation (9) in \cite{Siegel}, for any $p \le X$ and $d$ a positive integer with $d^2 \mid 1-4p$, 
\begin{equation}\label{Siegel partial sum}
\sum_{n > B} \Kronecker{(1-4p)/d^2}{n} \frac{1}{n} = O(B^{-1} X^{1/2} \log X).
\end{equation}
Summing over all $p \le X$ with $d^2 \mid 1-4p$,
\begin{equation}\label{tail bound}
\begin{split}
T_d(X, [B,\infty]) &= \sum_{\substack{p \le  X\\ d^2 \mid 1-4p}} \sum_{n \ge B} \Kronecker{(1-4p)/d^2}{n} \frac{1}{n}   \\
= O(B^{-1} X^{3/2}),
\end{split}
\end{equation}
since we have $O(X/\log X)$ terms each of which is $O(B^{-1} X^{1/2} \log X)$ by \eqref{Siegel partial sum}.\\

We now choose a second cutoff $b < B$, which we'll choose small enough that we can estimate $T_d(X, b)$ by Siegel--Walfisz, and then will estimate by breaking up
\[
T_d(X, [1, B]) = T_d(X, [1, b]) + T_d(X, [b, B]).
\]

We will estimate $T_d(X, [1, b])$ by Siegel--Walfisz and $T_d(X, [b, B])$ using our unrestricted bilinear character sum bound.

\subsection{Estimating $T_d(X, [1, b])$}

We exchange order of summation, to get
\begin{equation}\label{change order}
T_d(X, [1, b]) = \sum_{1 \le n \le b} \frac{1}{n} \sum_{\substack{p \le  X\\ d^2 \mid 1-4p}}  \Kronecker{(1-4p)/d^2}{n} .
\end{equation}
We now estimate the inner sum.  Note that  $\Kronecker{(1-4p)/d^2}{n}$ depends only on $p$ modulo $4d^2 n = O(b (\log X)^{2 \alpha})$ as we are in the ranges $n \le b$ and $d \le (\log X)^{\alpha}$.  We apply Siegel--Walfisz in the form of Lemma \ref{S-W new} to get 
\begin{equation}\label{use siegel-walfisz}
\begin{split}
\sum_{\substack{p \le  X\\ d^2 \mid 1-4p}}   \Kronecker{(1-4p)/d^2}{n} =  \frac{1}{\phi(nd^2)} \sum_{\substack{a \in (\bZ/d^2 n\bZ)^\ast \\ d^2 \mid 1-4a} } \Kronecker{(1-4a)/d^2}{n} \Li(X) + O(d^2 n X (\log X)^{-A})\\ 
= (\phi(d^2)^{-1}) \A(n, d) \Li(X) + O(X b (\log X)^{2 \alpha -A})
\end{split}
\end{equation}
for any $A > 0$, where the arithmetic function $\A(n, d)$ was defined in \eqref{Define A(n,d)}.\\



Now plugging \eqref{use siegel-walfisz} into \eqref{change order}, we obtain
\begin{equation}\label{summed s-w}
    T_d(X,[1, b]) =   \phi(d^2)^{-1} \left(\sum_{1 \le n \le b} \frac{\A(n, d)}{n}\right) \Li(X) +  O\left( (b \log b) X (\log X)^{2\alpha -A}\right)
\end{equation}

We now use the the expression for $\sum_n \A(n, d) n^{-s}$ given in \eqref{Dirichlet}, and set $s =1$.  We deduce that the sum $\sum_{n \le b} \A(n, d) n^{-1}$  converges absolutely as $b \to \infty$, with sum 
\begin{equation} \label{andsum}
\begin{split}
\sum_{n} \frac{\A(n, d)}{n} = f_d(1) 
&= \frac{1}{2} \zeta(2) \prod_{\ell \nmid d \text{ odd}} \left( 1 - \left(\frac{1}{\ell-1}\right) ( \ell^{-1} + \ell^{-2}) \right) \prod_{\ell \mid d} (1- \ell^{-3})\\
&= \frac{1}{2} \zeta(2) \prod_{\ell \mid d} \left(\frac{\ell-1}{\ell}\right) \left(\frac{\ell^3-1}{\ell^3 - \ell^2 - \ell - 1} \right) \prod_{\ell \text{ odd}} \left( \frac{\ell^3 - \ell^2 - \ell - 1}{\ell^3 - \ell^2} \right) \\
&= \frac{1}{2} \zeta(2) d^3 c(d) \prod_{\ell \mid d} \left(\frac{\ell-1}{\ell}\right) \prod_{\ell \text{ odd}} \left( \frac{\ell^3 - \ell^2 - \ell - 1}{\ell^3 - \ell^2} \right),
\end{split}
\end{equation}
using the definition of $c(d)$ in \eqref{define c}.\\

To bound $\sum_{n >b} \A(n, d) n^{-1}$, we use
\[
\sum_{n > b} \frac{\A(n, d)}{n} < b^{-1/2 + \ep} \sum_{n > b} A(n, d) n^{-1/2 - \ep} = O_\ep(b^{-1/2 + \ep})
\]
for any $\ep > 0$ since the sum $\sum_{n \ge 1} A(n, d) n^{-1/2-\ep}$ converges absolutely. It follows that

\begin{equation}
\sum_{1 \le n \le b} \frac{\A(n, d)}{n} = \frac{1}{2} \zeta(2) d^3 c(d) \prod_{\ell \mid d} \left(\frac{\ell-1}{\ell}\right) \prod_{\ell \text{ odd}} \left( \frac{\ell^3 - \ell^2 - \ell - 1}{\ell^3 - \ell^2} \right) + O_\ep(b^{-1/2 + \ep})
\end{equation}

Substituting this into (\ref{summed s-w}), we obtain 
\begin{equation}\label{with explicit error term}
\begin{split}
    T_d(X, [1, b]) = & \phi(d^2)^{-1} \left(\frac{1}{2} \zeta(2) d^3 c(d) \prod_{\ell \mid d} \left(\frac{\ell-1}{\ell}\right) \prod_{\ell \text{ odd}}  \frac{\ell^3 - \ell^2 - \ell  - 1}{\ell^3 - \ell^2} \right) \Li(X) + O_\epsilon(b^{-1/2 + \ep} \Li(X))  \\
    &+ O_A \left(\frac{X (b \log b)}{(\log X)^{A - 2\alpha}} \right) \\
= & \frac{\pi^2}{12} d c(d)  \prod_{\ell \mid d} \left( \frac{\ell^3 - \ell}{\ell^3 - \ell^2- \ell - 1} \right) \Li(X) + O_{A, \ep} \left( \max \left(b^{-1/2 + \ep} \Li(X), \frac{X (b \log b)}{(\log X)^{A- 2\alpha}} \right) \right)
\end{split}
\end{equation}

\subsection{Estimating $T_d(X, [b, B])$}

We will now apply our unrestricted bilinear character sum bound (Lemma~\ref{unrestricted double oscillation}) to
\[
T_{d}(X, [b, B]) = \sum_{\substack{p \le  X\\ d^2 \mid 1-4p}} \sum_{b < n \le B} \Kronecker{(1-4p)/d^2}{n} \frac{1}{n}.
\]

Following Friedlander--Iwaniec in \cite{FI hyperbolic}, we divide the interval $[b, B]$ into dyadic intervals $[a, 2a]$ and apply Lemma~\ref{unrestricted double oscillation} on each.

\begin{lemma}\label{double cancellation interval}
For $a \ll X^{1/2}$,
\[
T_{d}(X, [a, 2a])
= \sum_{\substack{p \le  X\\ d^2 \mid 1-4p}} \sum_{a \le n \le 2a} \Kronecker{(1-4p)/d^2}{n} \frac{1}{n} \ll X (\log X) \cdot a^{-1/2}
\]
\end{lemma}
\begin{proof}
We rewrite
\[
T_{d}(X, [a, 2a]) = \frac{1}{a} \sum_{\substack{p \le  X\\ d^2 \mid 1-4p}} \sum_{a \le n \le 2a} \frac{a}{n} \Kronecker{(1-4p)/d^2}{n} .
\]
so all the coefficients in the inner sum are $1$-bounded, and  Lemma~\ref{unrestricted double oscillation} applies to give
\[
T_{d}(X, [a, 2a]) \ll \frac{1}{a} \cdot a X (\log a + \log X) (a^{-1/2} + (X/a)^{-1/2}) \ll X \log X \cdot a^{-1/2},
\]
since $a \ll X^{1/2}$.
\end{proof}

We now apply this lemma to bound $T_{d}(X, [b, B])$ as follows:

\begin{equation}\label{intermediate range}
T_{d}(X, [b, B]) \le \sum_{i = 0}^{\lfloor \log_2(B/b) \rfloor -1} T_{[2^i b, 2^{i+1}b]}(X) \ll  X \log X \cdot b^{-1/2} \sum_{i = 0}^{\lfloor \log_2(B/b) \rfloor -1} 2^{-i/2} \ll X \log X \cdot  b^{-1/2}.
\end{equation}

Since we have chosen $b = (\log X)^{\beta}$ for $\beta > 4$, this is an error term. 
\subsection{ Proof of Lemma~\ref{estimate $T_d$}}
We are now ready to put all of our above estimates together.
\begin{proof}[Proof of Lemma~\ref{estimate $T_d$}]

Summing \eqref{with explicit error term}, \eqref{intermediate range}, and \eqref{tail bound}
\begin{equation}
\begin{split}
    T_d(X) = & T_d(X, [1, b])  + T_d(X, [b, B]) + T_d(X, [B, \infty] \\
    = & \frac{\pi^2}{12} d c(d)  \prod_{\ell \mid d} \left( \frac{\ell^3 - \ell}{\ell^3 - \ell^2- \ell - 1} \right) \Li(X)+ O_{A, \epsilon} \left( \max \left(b^{-1/2 + \epsilon} \Li(X), \frac{X (b \log b)}{(\log X)^{A- 2\alpha}} \right) \right)\\ &+ O(X \log X \cdot  b^{-1/2})
    + O(B^{-1} X^{3/2}).
\end{split}
\end{equation}

Setting $\ep = 1/4$,  $b = (\log X)^{4C}$, $A =   2 \alpha + 5C+1$, and $B = X^{1/2} \log(X)^C$, for a positive real parameter $C$, our error term takes the form 
\[
O_{C} \left(\max \left\{ \frac{X}{(\log X)^{C+1}} , \frac{X \log(X)^{4C} \log \log X}{\log(X)^{5C+1}}, \frac{X}{ \log(X)^{2C - 1}}, \frac{X}{(\log X)^{C}} \right\}\right) 
\]
which is $O_C(X (\log X)^{-C})$ for $C > 1$. 
\end{proof}

\end{document}